\documentclass{amsart}
\usepackage{amsfonts}
\usepackage{amsmath,amssymb}
\usepackage{amsthm}
\usepackage{amscd}
\usepackage{graphics}
\usepackage{graphicx}

\theoremstyle{remark}{
\newtheorem{Def}{{\rm Definition}}

}
\theoremstyle{plain}
{

\newtheorem{Prop}{Proposition}
\newtheorem{Thm}{Theorem}

}

\begin{document}
\title[Generic planar graphs and Reeb graphs of real algebraic functions]{Planar graphs embedded in generic ways and realizing them as Reeb graphs of real algebraic functions }
\author{Naoki kitazawa}

\address{Osaka Central Advanced Mathematical Institute (OCAMI) \\
3-3-138 Sugimoto, Sumiyoshi-ku Osaka 558-8585
TEL: +81-6-6605-3103
}
\email{naokikitazawa.formath@gmail.com}
\urladdr{https://naokikitazawa.github.io/NaokiKitazawa.html}
\maketitle
\begin{abstract}
This paper is concerned with long-time interest of us, especially, the author, in realizing graphs as Reeb graphs of real algebraic functions of certain nice classes.

The {\it Reeb graph} of a differentiable function is the set consisting of all components of preimages of all single points and endowed with the quotient topology canonically. In tame cases, such objects are graphs. The Reeb graph of the natural height of the unit sphere of dimension at least $2$ is a graph with exactly one edge and homeomorphic to a closed interval. These graphs have been fundamental and strong tools in geometry since theory of Morse functions has been established in the former half of the last century. 

We present a new answer to the problem, saying that generically embedded planar graphs are homeomorphic to the Reeb graphs of real algebraic functions obtained by elementary polynomials and elementary procedures. 

\end{abstract}
\section{Introduction.}
\label{sec:1}
The paper \cite{kitazawa2} is a pioneering study motivating us, including the author himself, to study whether given graphs are Reeb graphs of real algebraic functions of certain nice classes. For the introduction or the present section, we respect the preprint of the author.
 
The {\it Reeb graph} of a differentiable function is the set of all connected components of the preimages of single points and endowed with the natural quotient topology of the manifold. The birth of theory of so-called {\it Morse} functions motivated us to introduce Reeb graphs. They are graphs in situations being not wild. They have been fundamental and strong tools in geometry, having some fundamental or important information of the manifolds. See \cite{reeb} and for a recent study, see \cite{saeki} for example.

Our problems on topological or combinatorial properties of Reeb graphs have been studied for differentiable functions since \cite{sharko}, followed first by \cite{masumotosaeki} and \cite{michalak}. The author has also contributed to this study: the author has respected not only graphs but also prescribed preimages of points (\cite{kitazawa1}). Excuse me we omit other related studies. 
 
Our study is motivated by this story and \cite{kitazawa2} is regarded as the first problem and its answer in the real algebraic situation and presented by the author himself.

Our study is also motivated by real algebraic geometry, especially, explicit construction of real functions.

Existence (and approximation) theory on real algebraic manifolds and maps has been established and has developed by Nash and Tognoli \cite{nash, tognoli}, mainly. \cite{kollar} surveys this. Construction and geography of such real algebraic objects is another problem, having been fundamental, natural and difficult.

We introduce elementary notions, and notation on manifolds and maps between manifolds.

\subsection{Manifolds and maps in differentiable (smooth) situations.}
We use ${\mathbb{R}}^k$ for the $k$-dimensional Euclidean space, being also the $k$-dimensional real vector space and the $k$-dimensional {\it real affine space}. It is also a Riemannian manifold with the so-called {\it standard Euclidean metric}, where for a point $x \in {\mathbb{R}}^k$, $x_j$ is the $j$-th component ($1 \leq j \leq k$) and for two points $x_1,x_2 \in {\mathbb{R}}^k$, $||x_1-x_2||:=\sqrt{{\Sigma}_{j=1}^k {({x_1}_j-{x_2}_j)}^2}$ denotes the distance of the two points under this metric. Let $||x||:=||x-0||$ with $0 \in {\mathbb{R}}^k$ being the origin. Let $D^k:=\{x \in {\mathbb{R}}^k \mid ||x|| \leq 1\}$, the $k$-dimensional unit disk, and $S^{k-1}:=\{x \in {\mathbb{R}}^k \mid ||x||=1\}$, the ($k-1$)-dimensional unit sphere.
The map ${\pi}_{m,n}:{\mathbb{R}}^m \rightarrow {\mathbb{R}}^n$ is the canonical projection, which maps $x=(x_1,x_2) \in {\mathbb{R}}^n \times {\mathbb{R}}^{m-n}={\mathbb{R}}^m$ to $x_1$ under the relation $m>n \geq 1$. A real polynomial map $c:{\mathbb{R}}^m \rightarrow {\mathbb{R}}^n$ is a map such that the $j$-th component $c_j:{\mathbb{R}}^m \rightarrow \mathbb{R}$ is defined by a real polynomial for arbitrary positive integers $m$ and $n$ and each integer $1 \leq j \leq n$. The canonical projection ${\pi}_{m,n}$ ($m \geq n$) is of simplest real polynomial maps. 

We use $T_x X$ for the tangent vector space of a differentiable manifold $X$ at a point $x \in X$. For a differentiable map $c:X \rightarrow Y$ between two differentiable manifolds $X$ and $Y$, let ${dc}_{x}:T_xX \rightarrow T_{c(x)} Y$ denote the differential at $x$. This is also a linear map. A {\it singular} point $x \in X$ of $c$ is a point such that the rank of the differential at $x$ ${dc}_x$ drops. Our differentiable maps here are smooth or of the class $C^{\infty}$), unless otherwise stated. A {\it diffeomorphism} means a homeomorphism which is smooth and has no singular point. Two smooth manifolds are {\it diffeomorphic} if a diffeomorphism between them exists.

A {\it Morse} function on a smooth manifold $X$ to a $1$-dimensional smooth manifold $C$ or a smooth curve is a smooth function $c:X \rightarrow C$ with no singular point on its boundary and with each singular point $p$ being always of the form $c(x_1,\cdots x_m)={\Sigma}_{j=1}^{m-i(p)+1}{x_j}^2-{\Sigma}_{j=1}^{i(p)} {x_{m-i(p)+j}}^2$ for suitable local coordinates and a suitable integer $0 \leq i(p) \leq m$.  See \cite{milnor} for Morse functions and see \cite{golubitskyguillemin} for related singularity theory of differentiable maps. A {\it Morse-Bott} function to $C$ is a smooth function of a certain generalized class of the class of Morse functions and defined as a function at each singular point which is represented as the composition of a smooth function with no singular point with a Morse function. For this, consult \cite{bott} for example. 
\subsection{Real algebraic objects.}
We introduce real algebraic objects. For this, check \cite{bochnakcosteroy, kollar} for example. We also respect our papers and preprints such as \cite{kitazawa2, kitazawa3, kitazawa4, kitazawa5, kitazawa9, kitazawa10}, where we do not assume the understanding.

A {\it non-singular} set (of ${\mathbb{R}}^m$) means a union of connected components of the zero set of a real polynomial map $c:{\mathbb{R}}^m \rightarrow {\mathbb{R}}^n$ is {\it non-singular} if the rank of the differential ${dc}_x$ does not drop at any point $x \in c^{-1}(0) \subset {\mathbb{R}}^m$. This comes from the implicit function theorem.

We call a non-empty non-singular set a {\it real algebraic} manifold. The real affine space and the unit sphere are of of such sets. Our {\it real algebraic} map is a map represented as the composition of the canonically defined embedding of a real algebraic manifold into the (underlying) real affine space with either the identity map or the canonical projection to another lower dimensional real affine space.

The real affine space, the unit sphere, the canonical embeddings of them to higher dimensional real affine spaces and their compositions with the canonical projections to lower dimensional real affine spaces give simplest examples of these cases.
\subsection{Our main result.}
Before presenting our main result, Theorem \ref{thm:1}, we explain graphs, their embeddings, and vertices of Reeb graphs, shortly. This causes no mathematical problems in the present paper. We explain these notions in the next section again.
A graph is a $1$-dimensional compact and connected CW complex the closure of each $1$-cell of which is homeomorphic to $D^1$. This is also finite. 

We can guess the definition of a {\it piecewise smooth} map easily, where we introduce related notions, more rigorously, later. We use the notion of piecewise smooth embedding here and this is no problem. An {\it embedding} of a graph is a {\it piecewise smooth} embedding into another smooth manifold. We also consider a so-called {\it generic} case with respect to the canonical projection ${\pi}_{2,1}$. An {\it M-generic} case is a kind of specific cases for the generic case.

The {\it Reeb graph} of a smooth function is explained shortly in the beginning. Reeb graphs are graphs in considerable situations. A point of a graph is a vertex of the graph if it corresponds to a component of the preimage of a single point containing some singular points of the function.


\begin{Thm}
\label{thm:1}
Let $G$ be a graph admitting an embedding into ${\mathbb{R}}^2$ which is  M-generic  with respect to ${\pi}_{2,1}$. 
We have a real algebraic function on some compact and connected real algebraic manifold which is Morse-Bott and represented as the composition of some real algebraic map as in Theorem \ref{thm:2}, presented later, with the canonical projection ${\pi}_{2,1}$, and whose Reeb graph is homeomorphic to $G$.
\end{Thm}
For realizing graphs as the Reeb graphs of real algebraic functions of certain classes, see \cite{kitazawa2} and see also our preprints  \cite{kitazawa3, kitazawa5, kitazawa6, kitazawa9, kitazawa10}. Especially, our new result explicitly respects \cite{kitazawa10} mainly. 
\subsection{The organization of our paper.}
In the next section, we explain {\it piecewise smooth} maps, {\it graphs}, {\it embeddings} of graphs, and {\it Reeb graphs} of smooth functions, again, and more precisely. We also review a kind of fundamental reconstruction of real algebraic maps onto regions in ${\mathbb{R}}^2$ surrounded by circles, first presented by the author (\cite{kitazawa3}), as Theorem \ref{thm:2}. After that, we prove our main result, or Theorem \ref{thm:1}. The third section is devoted to our additional remark.

\section{On our main result.} 
Hereafter, a {\it straight line} means a real algebraic manifold in ${\mathbb{R}}^2$ which is also the zero set of a real polynomial of degree $1$. It is diffeomorphic to $\mathbb{R}$. A {\it straight segment} means a subspace of a straight line which is diffeomorphic to $D^1$.
An {\it ellipse} {\it centered at} $p$ means the boundary of the set of the form $\{(x_1,x_2) \mid a_1{(x_1-p_1)}^2+a_2{(x_2-p_2)}^2=r\} \subset {\mathbb{R}}^2$ for some $a_1,a_2,r>0$ and some ordered pair $p=(p_1,p_2)$ of real numbers.
A {\it circle} centered at $p$ and of radius $\sqrt{r}>0$ is a specific case of ellipses defined by the additional condition $a_1=a_2=1$. The point $p$ is the {\it center} of the circle (ellipse). 

They are real algebraic manifolds and diffeomorphic to $S^1$.

We assume some fundamental knowledge and arguments on the zero sets of real polynomials of degree $2$ in ${\mathbb{R}}^2$ or real plane curves of degree $2$ and quadratic real plane curves.

We also assume elementary knowledge on Euclidean plane geometry. Especially, elementary arguments on intersections of shapes such as straight lines, segments, and circles, in the Euclidean plane ${\mathbb{R}}^2$, are implicitly applied in several scenes. As another notion, we need to understand elementary arguments on the "similarity" of shapes.
\subsection{Graphs, embeddings of graphs, and differentiable functions and their Reeb graphs, again.}
We review CW complexes and PL topology, shortly.

We omit precise exposition on fundamental terminologies and propositions on CW complexes. 

A topological manifold has the structure of a CW complex. A smooth manifold has the structure of a CW complex canonically, which is a so-called {\it PL manifold}. Hereafter, unless otherwise stated, a smooth manifold is seen as a PL manifold when it is seen as a CW complex.

A {\it piecewise smooth} map $c:X \rightarrow Y$ on a CW complex $X$ into another CW complex $Y$ means a continuous map which is, for each cell $e_X$ of the CW complex and some subset $F_{e_X} \subset e_X$ whose Lebesgue measure can be calculated as $0$ in $e_X$, at each point $p_{e_X} \in e_X-F_{e_X}$, seen as a smooth map into some cell $e_{p,Y}$. A {\it piecewise smooth} embedding means a piecewise smooth map being also a topological embedding. A piecewise smooth homeomorphism is defined as a piecewise smooth map being also a homeomorphism and a {\it piecewise smooth} isotopy means an isotopy considered in the topology category which is also a piecewise smooth map.

An {\it isomorphism} between CW complexes means a piecewise smooth homeomorphism mapping each cell of a CW complex onto some cell in another CW complex and two CW complexes are {\it isomorphic} if an isomorphism exists between them. 

The definitions of piecewise smooth maps of the presented classes do not depend on the structures of the CW complexes.
   
For a union of cells of a CW complex, a {\it regular neighborhood} of it is defined. We do not explain its rigorous definition and in short, it is a neighborhood which is a closed set in the CW complex and collapses to the union.

Here, our {\it graph} is a $1$-dimensional connected, compact, and finite CW complex and an {\it edge} (a {\it vertex}) of it is its $1$-cell (resp. $0$-cell). 
We do not admit our graphs having edges whose closures are homeomorphic to $S^1$: the closures of edges of our graphs are always homeomorphic to $D^1$. We admit a {\it multigraph}: a multigraph may have a pair of distinct vertices connected by at least two distinct (closures of) edges.
The {\it edge} ({\it vertex}) {\it set} of the graph is the set of all edges (resp. vertices) of it.
The {\it degree} of a vertex $v$ of the graph $G$ is the number of all edges of it containing $v$.

As a specific case of CW complexes, an isomorphism between two graphs is a piecewise smooth homeomorphism mapping the vertex set of a graph onto that of the other and two graphs are {\it isomorphic} if there exists an isomorphism between them.

An {\it embedding} of a graph into a smooth manifold means a piecewise smooth embedding of the graph.

An embedding $e_G:G \rightarrow {\mathbb{R}}^2$ of a graph is {\it generic with respect to} ${\mathbb{R}}^2$ if the composition of this with the projection ${\pi}_{2,1}$ is a piecewise smooth embedding on each edge of the graph. The embedding is {\it M-generic with respect to} the projection if for every vertex $v$ of degree at least $2$ the image of some small open neighborhood $N(v)$ of $v$ by the composition of the embedding $e_G$ with the projection contains the value at $v$ in the interior of the image considered in $\mathbb{R}$.

The following is a kind of exercises on so-called PL topology and usage of approximation. 

\begin{Prop}
	\label{prop:1}
For an embedding $e_G:G \rightarrow {\mathbb{R}}^2$ of a graph $G$ being generic with respect to the projection ${\pi}_{2,1}$, there exists piecewise smooth isotopy
$E_{G}:G \times \{t \mid 0 \leq t \leq 1\} \rightarrow  {\mathbb{R}}^2$ satisfying the following.
\begin{enumerate}
\item The restriction of $E_{G}$ to $G \times \{0\}$ is $e_G$ where $G$ and $G \times \{0\}$ are identified by the map mapping $p \in G$ to $(p,0) \in G \times \{0\}$.
\item The restrictions of $E_{G}$ to the sets $e \times \{t\}$  {\rm (}$0 \leq t \leq 1${\rm )} are always generic with respect to the projection.
\item At each vertex $v$ of $G$, we have the values ${\pi}_{2,1} \circ E_{G}(v,t)$ {\rm (}$0 \leq t \leq 1${\rm )} as a constant value.
\item On each edge of $e$ of $G$, the restrictions of ${\pi}_{2,1} \circ E_{G}$ to the sets $e \times \{t\}$  {\rm (}$0 \leq t \leq 1${\rm )} are always piecewise smooth embeddings.
\item The image $E_{G}(G \times \{1\})$ is a union of straight segments in ${\mathbb{R}}^2$ where for any distinct pair of these straight segments, they have exactly one point in common or do not intersect.
\end{enumerate} 
As a specific case, in the case where an embedding $e_G$ is M-generic with respect to the projection, the restrictions of $E_{G}$ to the sets $e \times \{t\}$  {\rm (}$0 \leq t \leq 1${\rm )} are always M-generic with respect to the projection ${\pi}_{2,1}$.
\end{Prop}

In short, a piecewise smooth isotopy of Proposition \ref{prop:1} preserves the values at vertices of the graphs of the functions represented as the compositions of the embeddings of the graphs with the projection ${\pi}_{2,1}$ and the (M-)genericity of the embeddings.
\begin{Def}
The {\it Reeb graph} of a smooth function $c:X \rightarrow C$ on a smooth manifold $X$ with no boundary to a $1$-dimensional smooth manifold $C$ with no boundary is defined in the following way.
Two points $x_1, x_2 \in X$ can be defined to be equivalent if and only if they are in a same connected component of a same preimage $c^{-1}(y)$. Let ${\sim}_c$ denote the equivalence relation on $X$. The quotient space $W_c:=X/{{\sim}_c}$ is defined as the {\it Reeb graph} of $c$ by the following rule under the conditions that $X$ is a closed manifold and that the image of the set of all singular points of $c$ is a finite set. 
Hereafter, $q_c:X \rightarrow W_c$ denotes the quotient map and the continuous function $\bar{c}:W_c \rightarrow C$ can be defined uniquely by the relation $c=\bar{c} \circ q_c$. The quotient space $W_c$ is a graph by defining its vertex $v$ as a point whose preimage ${q_c}^{-1}(v)$ contains some singular points of $c$. 
In this situation, $\bar{c}$ is a smooth function with no singular point on each edge of $W_c$.
\end{Def}
A main ingredient of the argument for the Reeb graph comes from a kind of classical results for functions of explicit nice classes such as the class of Morse(-Bott) functions. See \cite{izar} and see also \cite{martinezalfaromezasarmientooliveira1, martinezalfaromezasarmientooliveira2}. \cite{saeki} presents some related general result.  

Note also that the quotient spaces here are defined for (continuous) maps between general topological spaces and that we can enjoy similar arguments around fundamental definitions.
%


\subsection{A fundamental theorem on reconstructing real algebraic maps onto regions surrounded by real algebraic hypersurfaces.}

The following is a kind of fundamental theorems in reconstructing real algebraic maps onto the closure of a given non-empty open set in a real affine space, first presented in \cite{kitazawa3}. This also extends an essential part of \cite{kitazawa2}, in which the case of manifolds $S_j$ being mutually disjoint has been investigated. For, see also a related preprint \cite{kitazawa4}, remarking \cite{kitazawa2}.
\begin{Thm}[\cite{kitazawa3}]
\label{thm:2}
Let $D \subset {\mathbb{R}}^2$ be an open set which is non-empty and the following are satisfied.
\begin{enumerate}
\item Let $\overline{D}$ denote the closure of $D$ considered in ${\mathbb{R}}^2$. The set $\overline{D}-D$ is a subset in the union ${\bigcup}_{j=1}^{l} S_j$ of $l>0$ circles $S_j$ {\rm (}$1 \leq j \leq l${\rm )} and $S_j \bigcap \overline{D}$ is non-empty for each $1 \leq j \leq l$.   
\item {\rm (}So-called transversality of the intersections of the circles{\rm )} If two distinct circles $S_{j_1} \subset {\mathbb{R}}^2$ and $S_{j_2} \subset  {\mathbb{R}}^2$ intersect in the subset $\overline{D} \subset {\mathbb{R}}^2$ and we choose a tangent vector which is not the zero vector of each tangent vector space there, then these two vectors are mutually independent. The intersection of distinct three circles $S_{j_1}, S_{j_2}, S_{j_3} \subset  {\mathbb{R}}^2$ and $\overline{D}$ is always empty.
\item We can choose a positive integer $l^{\prime}$ with the following properties.
\begin{enumerate}
\item An integer $m_{l}(j)$ satisfying the relation $1 \leq m_{l}(j)=j^{\prime} \leq  l^{\prime}$ can be defined according to the rule that if $S_{j_1} \bigcap S_{j_2}$ and $\overline{D}$ intersect in a non-empty set, then $m_{l}(j_1)$ and $m_{l}(j_2)$ are not equal.
\item For each integer ${j_0}^{\prime}$ with $1 \leq {j_0}^{\prime} \leq l^{\prime}$, there exists at least one integer $j_0$ satisfying $m_{l_1}(j_0)={j_0}^{\prime}$. A positive integer $I_{j^{\prime}}$ is corresponded to each integer $1 \leq j^{\prime} \leq l^{\prime}$. 
\end{enumerate}
\end{enumerate}
Here, we have a real algebraic manifold $M:=\{(x,\{y_{I,j^{\prime}}\}_{j^{\prime}=1}^{l^{\prime}}) \in {\mathbb{R}}^2 \times {\prod}_{j^{\prime}=1}^{l^{\prime}} {\mathbb{R}}^{I_{j^{\prime}}+1}={\mathbb{R}}^{{\Sigma}_{j^{\prime}=1}^{l^{\prime}} (I_{j^{\prime}}+1)+2} \mid {\prod}_{j_I \in \{j \mid m_{l_1}(j)=j^{\prime \prime}\}} (f_{j_I}(x))-||y_{I,j^{\prime \prime}}||^2=0, 1 \leq j^{\prime \prime} \leq l^{\prime}\}$ under the following conditions.
\begin{enumerate}
	\setcounter{enumi}{3}
	\item A polynomial $f_j$ is of degree $2$ and $S_j:=\{x \in {\mathbb{R}}^2 \mid f_j(x)=0\}$. 
	\item The relation $\overline{D}=\{x \mid f_j(x) \geq 0\}$ holds.
	\end{enumerate}
	 Furthermore, the set $M$ is the zero set of the {\rm (}canonically defined{\rm )} real polynomial map and of dimension ${\Sigma}_{j^{\prime}=1}^{l^{\prime}} (I_{j^{\prime}}+1)+2$. In addition, the restriction of ${\pi}_{{\Sigma}_{j^{\prime}=1}^{l^{\prime}} (I_{j^{\prime}}+1)+2,2}$ there is a real algebraic map with the image being $\overline{D}$.
\end{Thm}

Note that originally, Theorem \ref{thm:2} is shown in a certain generalized situation. For example, in \cite{kitazawa10}, regions surrounded by such circles and straight lines appear explicitly.

We use $\mathbb{N} \subset \mathbb{R}$ for the set of all positive integers and ${\mathbb{N}}_{t} \subset \mathbb{N}$ for that of all positive integers smaller than the real number $t$.

What follows here, is from a kind of exercises on Morse functions and fundamental singularity theory. The function ${\pi}_{m+1,1} {\mid}_{S^m}$ is a Morse function with exactly two singular points and a Morse function for Reeb's (celebrated) theorem (\cite{reeb}). The restriction of this function to an $m$-dimensional (sub)manifold $X \subset S^m$ whose singular points are always those of the original function and in the interior of $X$ is considered here. Such a function is referred to as the {\it function of the natural height} of $X$.  

\begin{proof}[Reviewing our proofs of Theorem \ref{thm:2} in a self-contained way]
We review our original proof of \cite{kitazawa3}, respecting \cite{kitazawa9, kitazawa10}. \cite{kitazawa9, kitazawa10} review the original proof in a self-contained way. 

We need fundamental arguments in real algebraic situations, discussed in \cite{kollar}, and differential topological singularity theory, explained in \cite{golubitskyguillemin}, for example.
A main ingredient is to investigate local preimages of a point $p \in \overline{D}$ for the resulting map on $M$ onto $\overline{D}$. We see that $M$ is locally a regular smooth submanifold of the underlying Euclidean space ${\mathbb{R}}^{{\Sigma}_{j^{\prime}=1}^{l^{\prime}} (I_{j^{\prime}}+1)+2}$ and that $M$ is also non-singular. \\
\ \\
Case 1.\ The case $p$ is in the open set $D$. \\
The preimage of a sufficiently small open neighborhood diffeomorphic to $D^2$ is regarded as a copy of $D^2 \times {\prod}_{j^{\prime} \in {\mathbb{N}}_{l^{\prime}}} S^{j^{\prime}}$. The map gives a smooth map with no singular point onto the $2$-dimensional disk, locally, for suitable local coordinates. Furthermore, the preimages of each point for the map are diffeomorphic to ${\prod}_{j^{\prime} \in {\mathbb{N}}_{l^{\prime}}} S^{j^{\prime}}$. \\
\ \\
Case 2.\ The case $p$ is in $\overline{D}-D$ and contained in exactly one circle $S_j$ from $\{S_j\}_{j=1}^l$. \\
The preimage of a sufficiently small open neighborhood of $p$ diffeomorphic to $D^2$ is regarded as a copy of $D^1 \times D^{I_{j^{\prime}}+1} \times {\prod}_{j^{\prime \prime} \in {\mathbb{N}}_{l^{\prime}}-\{j^{\prime}\}} S^{I_{j^{\prime \prime}}}$. This produces the product map of a Morse function which is seen as the natural height of the disk and the identity map on the product of a copy of $D^1$ and finitely many copies of the unit spheres, locally, for suitable local coordinates. \\
\ \\
Case 3.\ The case $p$ is in $\overline{D}-D$ and contained in exactly two distinct circles $S_{j_1}$ and $S_{j_2}$ from $\{S_j\}_{j=1}^l$. \\
The preimage of a sufficiently small open neighborhood of $p$ diffeomorphic to $D^2$ is regarded as a copy of $D^{I_{{j_1}^{\prime}}+1} \times D^{I_{{j_2}^{\prime}}+1} \times {\prod}_{j^{\prime \prime} \in {\mathbb{N}}_{l^{\prime}}-\{{j_1}^{\prime},{j_2}^{\prime}\}} S^{I_{j^{\prime \prime}}}$. This produces the product map of two Morse functions each of which is seen as the natural heights of the disk and the identity map on the product of finitely many copies of the unit spheres, locally, for suitable local coordinates. \\

From the definition, $M$ is a real algebraic manifold and the zero set of a naturally defined real polynomial map. This completes the proof.
\end{proof}
\subsection{Proving Theorem \ref{thm:1}.}
We present two consecutive important propositions. We omit a precise presentation on the proof. For the proof, the method in the paper \cite{bodinpopescupampusorea} is also respected. \cite{bodinpopescupampusorea} obtains a region in ${\mathbb{R}}^2$ from a given graph in ${\mathbb{R}}^2$ surrounded by mutually disjoint $1$-dimensional real algebraic manifolds and canonically collapsing to a given graph.
\begin{Prop}
\label{prop:2}
In Proposition \ref{prop:1}, we have a small regular neighborhood $N(E_G(G \times \{1\})) \subset {\mathbb{R}}^2$ of the graph $E_G(G \times \{1\})$ where we discuss in the PL category. We have this in such a way that its boundary is a union of straight segments and homeomorphic to a disjoint union of copies of the circle $S^1$. We can also choose this in such a way that the restriction of ${\pi}_{2,1}$ to each connected component of the boundary of $N(E_G(G \times \{1\}))$ is generic with respect to ${\pi}_{2,1}$, where the circle is seen as a graph, canonically. This is not M-generic with respect to the projection. 
\end{Prop}
\begin{Prop}
\label{prop:3}
In Proposition \ref{prop:2}, we can also have the objects with the following properties by adding vertices to the original graph $G$ suitably.
\begin{enumerate}
\item For each connected component of the boundary of $N(E_G(G \times \{1\}))$, seen as a graph, two adjacent edges $e_{v,1}$ and $e_{v,2}$ of it having exactly one vertex $v$ of it in common, the closure of exactly one of these two edges of it is a straight segment in the straight line represented by the form $\{(t,y_v) \mid t \in \mathbb{R}\}$ for a suitable number $y_v$ and that of the other is not.  
\item We have the following quotient space with the structure of a graph as we have defined the Reeb graph of a smooth function.
\begin{enumerate}
\item The quotient space $Q_{{\pi}_{2,1}}(N(E_G(G \times \{1\})))$of $N(E_G(G \times \{1\}))$ is defined in the same way as that in defining the Reeb graph of a smooth function on a manifold with no boundary by using he restriction of ${\pi}_{2,1}$. Let $q_{{\pi}_{2,1},N(E_G(G \times \{1\}))}$ denote the quotient map.
\item The quotient space has the structure of a graph by defining its vertices $v$ as points whose preimages ${q_{{\pi}_{2,1},N(E_G(G \times \{1\}))}}^{-1}(v)$ contain some vertices of the boundary of $N(E_G(G \times \{1\}))$. 
\end{enumerate}
This graph is called a {\rm Poincar\'e-Reeb graph} of the regular neighborhood $N(E_G(G \times \{1\}))$, respecting \cite{bodinpopescupampusorea}. We can define similar objects and do similar discussions for similar situations and suitably generalized situations.
\item The Poincar\'e-Reeb graph of $N(E_G(G \times \{1\}))$ collapses to a graph homeomorphic to $E_G(G \times \{1\})$. More precisely, for each vertex of graph $G$ of degree greater than $3$, exactly one edge of the graph $Q_{{\pi}_{2,1}}(N(E_G(G \times \{1\})))$ is corresponded to, this yields a one-to-one correspondence and by collapsing these edges of the graph, we have a graph homeomorphic to $G$ and $E_G(G \times \{1\})$.
\end{enumerate}
\end{Prop}
\begin{proof}[A proof of Theorem \ref{thm:1}.]
A main ingredient of our strategy is construction of a region $D$ as in Theorem \ref{thm:2}, which is surrounded by sufficiently small circles $S_j$ centered at suitable points and of fixed radii in the interior of $N(E_G(G \times \{1\})) \subset {\mathbb{R}}^2$. Moreover, we need to choose the piecewise smooth isotopy and $N(E_G(G \times \{1\})) \subset {\mathbb{R}}^2$ of Propositions \ref{prop:2} and \ref{prop:3}, suitably. \\

As presented before, we need fundamental arguments on Euclidean geometry. For example, arguments on intersections of shapes are essential, and for sufficiently small objects, we implicitly use arguments on similarity of shapes. \\
\ \\
STEP 1  Choosing a piecewise smooth isotopy $E_G$ and the set $N(E_G(G \times \{1\})) \subset {\mathbb{R}}^2$ in Propositions \ref{prop:2} and \ref{prop:3}. \\

We can choose the piecewise smooth isotopy $E_G$ and the set $N(E_G(G \times \{1\})) \subset {\mathbb{R}}^2$ in Propositions \ref{prop:2} and \ref{prop:3} in such a way that the following hold.
\begin{itemize}
	\item For each vertex $v_{E_G}$ of the graph $E_G(G \times \{1\})$ which is not of degree $2$, we have a neighborhood $N(v_{E_G}) \subset {\mathbb{R}}^2$ of $v_{E_G}$ which is diffeomorphic to $D^2$ and whose boundary is a circle centered at $v_{E_G}$ and of a fixed sufficiently small radius.
\item For each connected component $E_{v_{E_G}}$ of $N(v_{E_G})-(N(v_{E_G}) \bigcap E_G(G \times \{1\}))$, there exists exactly one vertex $v_{E_{v_{E_G}}}$ of degree $2$ of a connected component $C_{E_{v_{E_G}}}$ of the boundary of $N(E_G(G \times \{1\}))$ where the value of the restriction of ${\pi}_{2,1}$ to the connected component $C_{E_{v_{E_G}}}$ of the boundary of $N(E_G(G \times \{1\}))$ is not in the interior of the image of any small open neighborhood $v_{E_{v_{E_G}}} \in U_{E_{v_{E_G}}} \subset C_{E_{v_{E_G}}}$ of $v_{E_{v_{E_G}}}$ considered in $C_{E_{v_{E_G}}}$. 
\item In addition, every vertex $v_{E_{v_{E_G}}}$ of degree $2$ of any connected component $C_{E_{v_{E_G}}}$ of the boundary of $N(E_G(G \times \{1\}))$ as above, or a vertex of $C_{E_{v_{E_G}}}$ where the value of the restriction of ${\pi}_{2,1}$ to a connected component $C_{E_{v_{E_G}}}$ of the boundary of $N(E_G(G \times \{1\}))$ is not in the interior of the image of any small open neighborhood $v_{E_{v_{E_G}}} \in U_{E_{v_{E_G}}} \subset C_{E_{v_{E_G}}}$ of $v_{E_{v_{E_G}}}$ considered in $C_{E_{v_{E_G}}}$, appears in this way. Thus we have a natural one-to-one correspondence.
 For this, important arguments of \cite{kitazawa8} are respected, where we do not assume related knowledge. Our related arguments here are regarded as a kind of variants of the original arguments.
\end{itemize}

\ \\
STEP 2 Choosing circles $S_j$. \\
We can choose a circle $S_j:=S_{v_{E_{v_{E_G}}}}$ centered at a suitable point and of a suitable sufficiently small radius for each vertex $v_{E_G}$ of the graph $E_G(G \times \{1\})$ and each connected component $C_{E_{v_{E_G}}}$ of $N(v_{E_G})-(N(v_{E_G}) \bigcap E_G(G \times \{1\}))$ with $v_{E_{v_{E_G}}} \in C_{E_{v_{E_G}}}$ in STEP 1 and satisfying the following.
\begin{itemize}
\item The circle $S_{v_{E_{v_{E_G}}}}$ bounds the disk $D_{v_{E_{v_{E_G}}}}$ containing the vertex $v_{E_{v_{E_G}}}$ of degree $2$ of the connected component $C_{E_{v_{E_G}}}$ in its interior.
\item In the case the vertex $v_{E_G}$ of the graph $E_G(G \times \{1\})$ is of degree at least $3$, the circle $S_{v_{E_{v_{E_G}}}}$ passes through a point of the straight line of the form $\{({\pi}_{2,1}(v_{E_G}),t) \mid t \in \mathbb{R}\}$ and at the point the tangent vector spaces of $S_{v_{E_{v_{E_G}}}}$ and the straight line, which are of $1$-dimensional, agree.
\item The set $S_{v_{E_{v_{E_G}}}} \bigcap N(E_G(G \times \{1\}))$ is diffeomorphic to $D^1$.
\item The disks $D_{v_{E_{v_{E_G}}}}$ are mutually disjoint.
\item  In the case the vertex $v_{E_G}$ of the graph $E_G(G \times \{1\})$ is of degree at least $3$, the restriction of ${\pi}_{2,1}$ to the 1-dimensional manifold $S_{v_{E_{v_{E_G}}}} \bigcap N(E_G(G \times \{1\}))$ has exactly one singular point and the unique point is in its interior and of the form $({\pi}_{2,1}(v_{E_G}),y_{v_{E_{v_{E_G}}}})$. In the case the vertex $v_{E_G}$ of the graph $E_G(G \times \{1\})$ is of degree $1$, the restriction of ${\pi}_{2,1}$ to the 1-dimensional manifold $S_{v_{E_{v_{E_G}}}} \bigcap N(E_G(G \times \{1\}))$ has no singular point. 
\end{itemize} 
See FIGURE \ref{fig:1} for example.

We can choose a circle $S_j$ centered at a suitable point and of a suitable sufficiently small radius as presented in FIGURE \ref{fig:2} for each vertex $v_{N(E_G(G \times \{1\}))}$ of $2$ of each connected component $C_{N(E_G(G \times \{1\}))}$ of the boundary of $N(E_G(G \times \{1\}))$ where the value of the restriction of ${\pi}_{2,1}$ to the connected component of the boundary of $N(E_G(G \times \{1\}))$ is in the interior of the image of a small open neighborhood $v_{N(E_G(G \times \{1\}))} \in U_{v_{N(E_G(G \times \{1\}))}} \subset C_{N(E_G(G \times \{1\}))}$ considered in $C_{N(E_G(G \times \{1\}))}$. We can also choose this as a circle bounding the disk $D_j$ containing $v_{N(E_G(G \times \{1\}))}$ in the interior.  

We can also choose remaining circles $S_j$ centered at suitable points and of suitable sufficiently small radii as presented in FIGURE \ref{fig:2} to have our desired region $D$. For the last circle $S_j$, we use $D_j$ for the disk bounded by $S_j$.
We can choose all chosen circles $S_j$ here satisfying the following conditions.
\begin{itemize}
\item The set $S_j \bigcap N(E_G(G \times \{1\}))$ is diffeomorphic to $D^1$.
\item The restriction of ${\pi}_{2,1}$ to the $1$-dimensional manifold $S_{j} \bigcap N(E_G(G \times \{1\}))$ has no singular point for circles being not chosen as $S_{v_{E_{v_{E_G}}}}$ in the present step.
\item The disks $D_j$ cover the boundary of of $N(E_G(G \times \{1\}))$.
\item Our desired region $D$ for Theorem \ref{thm:2} is obtained by putting "circles $S_j$ in Theorem \ref{thm:2}" as the chosen circles $S_j$. Here, our region $D$ is also a subset of the interior of $N(E_G(G \times \{1\}))$.   
\end{itemize} 

\begin{figure}
		\includegraphics[width=70mm,height=65mm]{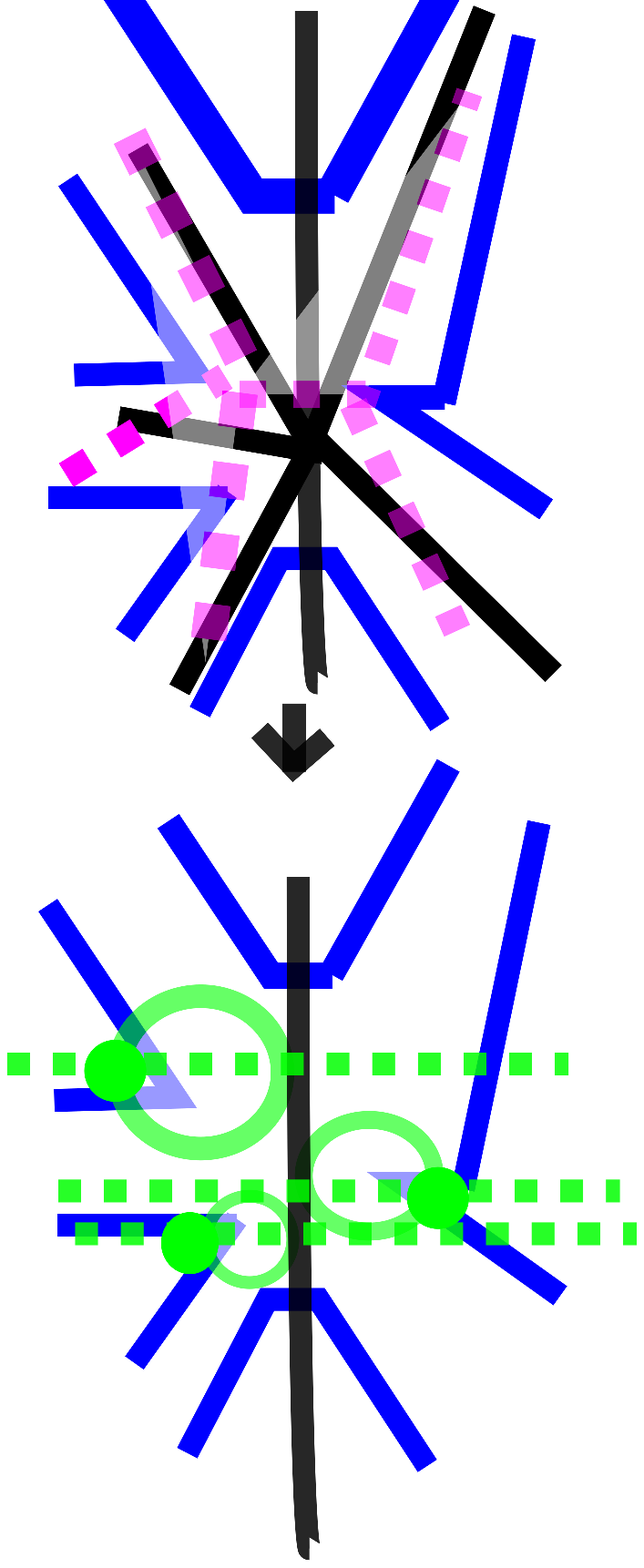}
		\caption{The black straight objects in the figure above show the graph $E_G(G \times \{1\})$ around a vertex $v_{E_G}$ of the graph $E_G(G \times \{1\})$ of degree $5$, the purple straight objects show the graph $Q_{{\pi}_{2,1}}(N(E_G(G \times \{1\})))$ locally, and the blue objects show the boundary of $N(E_G(G \times \{1\}))$, which is also a graph. In the figure below, the green circles show important circles $S_j:=S_{v_{E_{v_{E_G}}}}$ and the green dotted lines show straight lines of the form $\{(t,y_0) \mid t \in \mathbb{R}\}$ passing through the centers of the circles, each green dot shows a point in a green circle and the dotted green straight line associated to the circle and outside $N(E_G(G \times \{1\}))$. The straight lines of the form $\{({\pi}_{2,1}(v_{E_G}),t) \mid t \in \mathbb{R}\}$ are depicted by the black straight lines. Such objects also show relative locations of the shapes.}
                      \label{fig:1}

		\end{figure}

\begin{figure}
		\includegraphics[width=70mm,height=65mm]{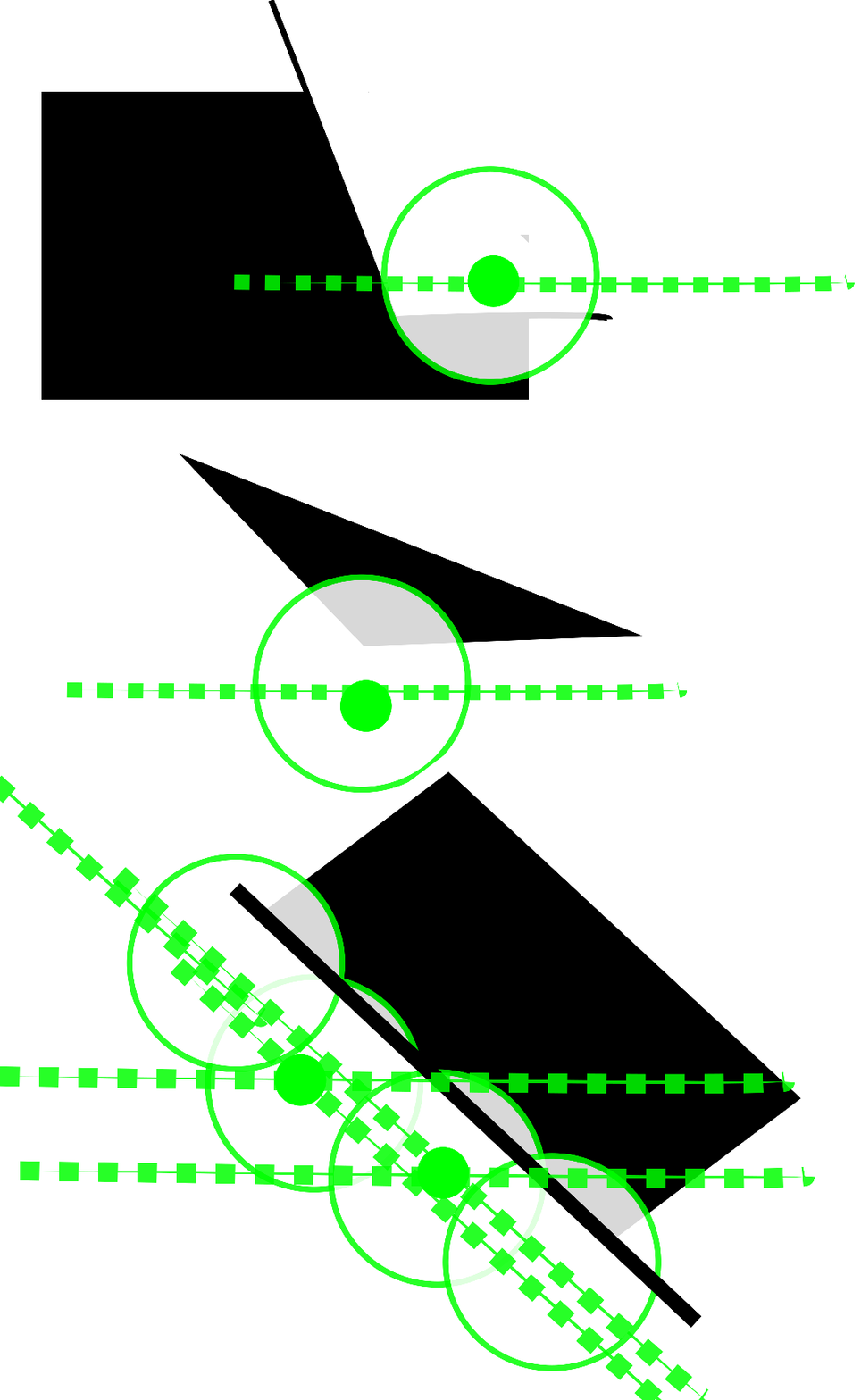}
		\caption{The former two figures here show circles $S_j$ for vertices $v_{N(E_G(G \times \{1\}))}$ of degree $2$ of connected components $C_{N(E_G(G \times \{1\}))}$ of the boundary of $N(E_G(G \times \{1\}))$ where the values of the restrictions of ${\pi}_{2,1}$ to the connected components of the boundary of $N(E_G(G \times \{1\}))$ are in the interiors of the images of some small open neighborhoods $v_{N(E_G(G \times \{1\}))} \in U_{v_{N(E_G(G \times \{1\}))}} \subset C_{N(E_G(G \times \{1\}))}$.
The last figure shows an explicit case for "the remaining circles $S_j$". In these present figures, objects which are colored in black or shaded are for $N(E_G(G \times \{1\}))$.
The green circles show important circles $S_j$ and the green dotted lines show straight lines of the form $\{(t,y_0) \mid t \in \mathbb{R}\}$ passing through the centers of the circles. Such objects show relative locations of the shapes.}
\label{fig:2}
		\end{figure}
By Theorem \ref{thm:2} with fundamental arguments on singularity theory of differentiable maps and fundamental arguments on real algebraic geometry, a desired real algebraic Morse-Bott function has been also constructed. More precise exposition is presented in the preprint \cite{kitazawa3, kitazawa9, kitazawa10}, where we do not assume the arguments there.

This completes our proof.
\end{proof}
\section{Our additional remark.}
\subsection{The Reeb graphs of functions represented as the compositions of maps as in Theorem \ref{thm:2} with ${\pi}_{2,1}$ in \cite{kitazawa3}.}
We have first investigated functions represented as the compositions of maps as in Theorem \ref{thm:2} with ${\pi}_{2,1}$ in \cite{kitazawa3}. \cite{kitazawa10} also shows this where we need a little effort and need to change regions there, surrounded by circles centered at several points and of fixed radii and straight lines, into regions as in Theorem \ref{thm:2}.
There, Reeb graphs are explicitly presented with the structures of the graphs. For example, \cite{kitazawa10} investigates cases of trees.
\subsection{Arrangements of circles in the plane ${\mathbb{R}}^2$.}
Motivated mainly by \cite{kitazawa3}, the author has started systematic studies of arrangements of circles, regions surrounded by them in the plane ${\mathbb{R}}^2$, and their shapes and combinatorics for example. For related pioneering studies, see \cite{kitazawa5, kitazawa6, kitazawa7, kitazawa8}. Our new case is not founded or studied explicitly in these studies.

 \section{Conflict of interest and data.}
The author is a researcher at Osaka Central
Advanced Mathematical Institute (OCAMI researcher). The institute is also supported by MEXT Promotion of Distinctive Joint Research Center Program JPMXP0723833165. The author is not employed there. However, the author thanks this.

Other than the present file for the present paper, no data are associated.

\end{document}